\pdfoutput=1
\documentclass{amsart}
\usepackage{gimac}
\usepackage{bbm}

\newcommand{\intzi}{\int^\infty_0}
\newcommand{\dx}{\, dx}
\newcommand{\dt}{\, dt}
\newcommand{\ds}{\, ds}
\newcommand{\dy}{\, dy}
\DeclareMathOperator{\sgn}{sign}

\begin{document}
  \title
    [Bose-Einstein condensation in Hyperbolic Kompaneets Equation]
    {Bose-Einstein condensation in a Hyperbolic model for the Kompaneets Equation}

  \author[Ballew]{Joshua Ballew}
  \address{
    Department of Mathematical Sciences,
    Carnegie Mellon University,
    Pittsburgh, PA 15213.}
  \email{jballew@andrew.cmu.edu}

  \author[Iyer]{Gautam Iyer}
  \address{%
    Department of Mathematical Sciences,
    Carnegie Mellon University,
    Pittsburgh, PA 15213.}
  \email{gautam@math.cmu.edu}

  \author[Pego]{Robert L. Pego}
  \address{%
    Department of Mathematical Sciences,
    Carnegie Mellon University,
    Pittsburgh, PA 15213.}
  \email{rpego@cmu.edu}

  \thanks{This material is based upon work partially supported by the National Science Foundation under grants
    DMS-1252912, 
    DMS-1401732, 
    and
    DMS-1515400. 
    GI also acknowledges partial support from an Alfred P. Sloan research fellowship.
    This work was partially supported by the Center for Nonlinear Analysis (CNA) under National Science Foundation PIRE Grant no.\ OISE-0967140.%
  }

  \keywords{Bose-Einstein condensation, Kompaneets equation}
  \subjclass[2010]{%
    35Q85, 
    35L04, 
    35L60.
  }

  \begin{abstract}
    In low-density or high-temperature plasmas, Compton scattering is the dominant process responsible for energy transport.
    Kompaneets in 1957 derived a non-linear degenerate parabolic equation for the photon energy distribution.
    In this paper we consider a simplified model obtained by neglecting diffusion of the photon number density
    in a particular way.
    We obtain a non-linear hyperbolic PDE with a position-dependent flux,  
    which permits a one-parameter family of stationary entropy solutions to exist.
    We completely describe the long-time dynamics of each non-zero solution, showing that it approaches some non-zero stationary solution.
    While the total number of photons is formally conserved, if initially large enough it necessarily decreases after finite time 
    through an out-flux of photons with zero energy.
    This corresponds to formation of a Bose-Einstein condensate, whose mass we show can only increase with time.
  \end{abstract}

  \maketitle
  \makeatletter
  \ifGI@draft\tableofcontents\fi
  \makeatother
  \section{Introduction}\label{sxnIntro}

  In low-density (or high-temperature) plasmas, \emph{Compton scattering} is the dominant process responsible for energy transport.
  In 1957, Kompaneets~\cite{Kompaneets57} derived an equation modeling this and was allowed to publish his work because it was considered useless for weapons research.
  Today Kompaneets' work has applications studying the interaction between matter and radiation in the early universe, the radiation spectra for the accretion disk around black holes, and various other fundamental phenomena in modern cosmology and high energy astrophysics~\cite{Birkinshaw99,SunyaevZeldovich70,ShakuraSunyaev89}.%

  In his work (see also~\cite{EscobedoMischler01}), Kompaneets derived a Fokker--Planck approximation for the Bolt\-zmann--Compton equation in the setting of a spatially uniform, isotropic, non-relativistic plasma at a constant temperature assuming the heat capacity of photons is negligible, and the dominant energy exchange mechanism is Compton scattering.
  In this setting Kompaneets showed that the evolution of the photon density $f$ is given by
  \begin{equation}\label{eqnKomp}
      \partial_t f
	= \frac{1}{x^2}
	    \partial_x \brak[\big]{ x^4 \paren{ \partial_x f + f + f^2 } },
  \end{equation}
  Here $x \in (0, \infty)$ and $t \geq 0$ are non-dimensionalized energy and time coordinates respectively.
  While the exact normalization in these coordinates is not important for the subsequent analysis, we remark that $x$ is proportional to the magnitude of the three dimensional photon wave vector.
  Consequently, $x$ is a radial variable, and the total number and total (non-dimensionalized) energy of the photons are given by
  \begin{equation*}
    N_f(t)
      \defeq \int_0^\infty x^2 f(x, t) \, dx
    \quad\text{and}\quad
    E_f(t)
      \defeq \int_0^\infty x^3 f(x, t) \, dx,
  \end{equation*}
  respectively.
  
  The boundary conditions associated to~\eqref{eqnKomp} are a little delicate.
  First, near $x = \infty$ it is natural to assume the incoming photon flux vanishes:
  \begin{equation}\label{eqnbc}
    \lim_{x \to \infty} x^4 \paren{ \partial_x f + f + f^2 } = 0.
  \end{equation}
  Near $x = 0$, the diffusion is degenerate and it is not clear \textit{a priori} whether a boundary condition can be imposed.
  We will revisit the boundary condition at $x=0$ later.

  Equation~\eqref{eqnKomp} formally possesses an entropy structure and dissipates the quantum entropy
  \begin{equation*}
    H_f(t) \defeq \int_0^\infty x^2 h(x, f(x, t)) \, dx,
  \end{equation*}
  where
  \begin{equation*}
    h(x, y) \defeq y \ln y - (1 + y) \ln ( 1 + y ) + xy.
  \end{equation*}
  Indeed, a direct calculation performed by Caflisch and Levermore  (see~\cite{CaflischLevermore86,LevermoreLiuEtAl16}) shows
  \begin{equation*}
    \partial_t H_f(t)
      = -\int_0^\infty x^4 f (1 + f)
      \brak[\big]{\partial_x (\partial_y h(x, f(x,t))) }^2 \, dx
      \leq 0.
  \end{equation*}
  This suggests that as $t \to \infty$, solutions to~\eqref{eqnKomp} approach an equilibrium $f(x)$ for which
  \begin{equation*}
    \partial_x ( \partial_y h(x, f(x) )) = 0.
  \end{equation*}
  All non-negative solutions of this are given by $\set{ f_\mu \mathop{\mid} \mu \geq 0 }$, where
  \begin{equation*}
    f_\mu(x) \defeq \frac{1}{e^{x + \mu} - 1 }.
  \end{equation*}

  This leads to an interesting conundrum.
  Multiplying~\eqref{eqnKomp} by $x^2$ and integrating shows that the total photon number is formally a conserved quantity.
  Since solutions to~\eqref{eqnKomp} are expected to approach one of the stationary solutions $f_\mu$ above, the total photon number should equal the total photon number of some equilibrium solution $f_\mu$.
  However,
  \begin{equation*}
    \sup_{\mu \geq 0} N_{f_\mu}
      = N_{f_0}
      = \int_0^\infty \frac{x^2}{e^x - 1} \, dx
      = 2 \zeta(3) < \infty,
  \end{equation*}
  and hence the total photon number in equilibrium is bounded above.
  Thus, if we start with more than $N_{f_0}$ photons, the total photon number will not be conserved.

Previous works by a number of authors suggest that a concentration of photons at
low energy can develop, and may cause an ``out-flux'' of photons near the $x = 0$ 
boundary~\cite{CaflischLevermore86,
EscobedoHerreroEtAl98,
EscobedoMischler01,JosserandPomeauEtAl06,
LevermoreLiuEtAl16,ZelDovichLevich69}. 
This is often interpreted as forming a Bose-Einstein condensate: a collection of
  zero-energy photons occupying the same quantum state.  
  While the existence of such condensates was
  predicted in 1924 by Bose and Einstein, they were only exhibited
  experimentally for photons in 2010 by Klaers et~al.~\cite{KlaersSchmittEtAl10},
  in circumstances dominated by physics different from Compton scattering.
Actually, the Kompaneets equation \eqref{eqnKomp} neglects physical effects,
such as Bremsstrahlung radiation, which may act to damp the low-energy spectrum
and suppress any out-flux at $x=0$.
Yet it remains interesting to investigate the behavior mathematically obtained
from the dynamics of the pure Kompaneets equation~\eqref{eqnKomp} in order
to understand how Compton scattering acts to create a photon flux toward low energy.
Following terminology developed in earlier works, we will refer to any out-flux
at $x=0$ as a contribution to a Bose-Einstein condensate at zero energy.

Mathematically, it was demonstrated by Escobedo et~al.~\cite{EscobedoHerreroEtAl98}
that there do exist solutions for which no-flux boundary conditions at $x=0$ break down
at some positive time, and an out-flux develops at this time. 
 Moreover, a unique global solution continues to exist subject to a boundedness condition 
for $x^2 f$ on $(0,1]$.  However, a complete mathematical understanding of the 
behavior of the Bose-Einstein condensate and the long-time dynamics of
solutions  of~\eqref{eqnKomp} is still unresolved.
  

  In an attempt to understand this problem better, many authors have studied simplified versions of~\eqref{eqnKomp}.
  In~\cite{ZelDovichLevich69}, the authors considered a hyperbolic model obtained by dropping the $\partial_x f$ term in~\eqref{eqnKomp}.
  This makes $x = 0$ an outflow boundary and the total photon number becomes a non-increasing function of time.
  This model, however, has no non-trivial stationary solutions, making the dynamics unphysical.

  In~\cite{KavianLevermore90} (see~\cite{Kavian02} for a published summary) the authors considered a linear model obtained by dropping the $f^2$ term in~\eqref{eqnKomp}.
  In this case, solutions dissipate an associated entropy and the stationary solutions correspond to the classical statics.
  However, the no flux boundary condition at $x = 0$ is automatically satisfied, without being imposed.
  Thus the total photon number is always conserved in time and no condensation can occur.

  Finally, in~\cite{LevermoreLiuEtAl16} the authors consider the non-linear Fokker-Planck equation obtained by dropping the linear $f$ term in~\eqref{eqnKomp}. This leads to dynamical behavior that is more like that which one expects for~\eqref{eqnKomp}.
  They show that solutions are uniquely determined without imposing a boundary condition near $x = 0$, and obtain a complete description of the long-time behavior.
  In particular, the authors show that the total photon number is non-increasing in time, and as $t \to \infty$ the solution converges to an equilibrium state of the form $f_\mu(x)=1/(x + \mu)$, for $\mu \geq 0$.
  However, because $x^2f_\mu(x)$ is unbounded, they work on the finite interval $x \in (0, 1)$ and impose a no-flux boundary condition at $x = 1$.

  In this paper, we consider a purely hyperbolic model obtained by rewriting~\eqref{eqnKomp} in terms of the number density $n = x^2 f$, and \textit{then} neglecting a diffusive term
  which was found in \cite{LevermoreLiuEtAl16} to have a negligible contribution to flux in the limit of small $x$.
  This results in a system that is quite attractive from a dynamical point of view, and does not have many of the deficiencies described above.
  Indeed, the system we obtain has an infinite family of localized stationary solutions, the largest of which asymptotically agrees with the classical Bose-Einstein statistics near $x = 0$.
  Further, for this system, every solution converges to some equilibrium solution as $t \to \infty$, and the total photon number is a decreasing function of time.
  Being hyperbolic, this system naturally allows a non-zero out-flux of photons at $x = 0$ corresponding to the formation of a Bose-Einstein condensate. 
We believe that the methods that we develop for our analysis may prove useful for study of the full Kompaneets equation and other 
models with related behavior.

  To derive the model we study, let~$n = x^2 f$ be the photon number density.
  Equation~\eqref{eqnKomp} now becomes
  \begin{equation}\label{eqnKn}
    \partial_t n =
      \partial_x \paren[\big]{  x^2 \partial_x n - 2xn + x^2 n + n^2}.
  \end{equation}
  Neglecting the dissipation term $x^2 \partial_x n$ in the flux gives us the model equation
  \begin{equation}\label{eqnSKomp}
    \partial_t n + \partial_x F = 0\,,
    \qquad 
    F(x, n) = (2x - x^2) n - n^2,
  \end{equation}
  on the domain $x > 0, t > 0$.
  From physical considerations we impose the boundary condition
  \begin{equation}\label{eqnSKompBC}
    F(x, n) \to 0
    \quad\text{as } x \to \infty.
  \end{equation}
  As we will see shortly, no boundary condition is required at $x = 0$.  For convenience, we will work with solutions initially having compact support in $x$. (This property is preserved 
  for all time $t>0$.)

  The system \eqref{eqnSKomp}--\eqref{eqnSKompBC} is a nonlinear hyperbolic problem with a position dependent flux.
  Following~\cite{Kruz}, it is natural to restrict our study to entropy solutions of this system.
  For clarity of presentation we postpone the definition of entropy solutions to Section~\ref{sxnExistence} and present our main results below.

  Our first result shows that \eqref{eqnSKomp}--\eqref{eqnSKompBC} admits a unique entropy solution, without imposing a boundary condition at $x = 0$.
  This solution approaches a stationary solution as $t \to \infty$, and the total photon number is non-increasing as a function of time.

  \begin{theorem}\label{thmMain}%
    For any non-negative, compactly supported initial data $n_0 \in L^1$, there exists a unique, non-negative, \emph{time global}, entropy solution to~\eqref{eqnSKomp}-\eqref{eqnSKompBC} such that
    \begin{equation}\label{eqnNormNBV}
      n \in L^\infty( [0, \infty), L^1[0, \infty) )
      \quad\text{and}\quad
      (1 - e^{-t}) n( \cdot, t ) \in L^\infty\paren[\big]{ [0, \infty), \operatorname{BV}( [0, \infty) ) }.
    \end{equation}
    Additionally, this solution (denoted by~$n$), satisfies the following properties:
    \begin{asparaenum}
      \item\label{thm1part1}
	There exists a unique $\alpha \in [0, 2]$ such that
	\begin{equation}\label{eqnNtoNAlpha}
	  \lim_{t \to \infty} \intzi \abs[\big]{ n(t, x) - \hat n_\alpha(x) } \, dx
	  = 0.
	\end{equation}
	Here $\hat n_\alpha$ are (all) the equilibrium entropy solutions, and are defined by
	\begin{equation}\label{eqnNAlpha}
	  \hat n_\alpha(x) =
	    \begin{dcases}
	      0	    & x \not\in (\alpha, 2),\\
	      2x - x^2    & x \in (\alpha, 2).
	    \end{dcases}
	\end{equation}

      \item\label{thm1part2}
	The total photon number
	\begin{equation*}
	  N[n(t, \cdot)] \defeq \int_0^\infty n( t,x) \, dx
	\end{equation*}
	 is a non-increasing function of time.
    \end{asparaenum}
  \end{theorem}

  Observe no boundary condition is imposed (or required) at the left endpoint~$x = 0$, and we will directly prove uniqueness of non-negative entropy solutions without any flux condition at $x = 0$.
  As we will see a possible out-flux can occur at $x = 0$ leading to a concentration of photons at zero energy (i.e.\ energy that is negligible on the scales described by the Kompaneets model).
  As remarked earlier, we interpret this out-flux 
  as a contribution to a Bose-Einstein condensate. Our result above shows that if the Bose-Einstein condensate forms, it can only increase in mass.

  We remark that the classical Bose-Einstein statistics postulate that the equilibrium photon energy distribution $x^2 f$ is
  \begin{equation*}
    x^2 f_\mu(x) = \frac{x^2}{e^{x+\mu} - 1}
  \end{equation*}
  for $\mu \geq 0$.
  Near the origin, $x^2 f_\mu$ is linear for $\mu = 0$ and quadratic for $\mu > 0$.
  All these solutions decay exponentially as $x \to \infty$.
Because we neglect the diffusion term in \eqref{eqnKn},
our equilibrium solutions no longer take this classical Bose-Einstein form.
  But they have similar asymptotic behavior for small $x$: $\hat n_\alpha$ is linear near the origin for $\alpha = 0$ and vanishes in the interval $[0, \alpha]$ for $\alpha > 0$.
  All our equilibrium solutions are compactly supported, and are identically $0$ for $x > 2$.

  Note Theorem~\ref{thmMain} only guarantees the total photon number is decreasing.
  We can, however, obtain a more precise description of this phenomenon. 
  \begin{proposition}\label{ppnLoss}
    If $n$ is a non-negative entropy solution to~\eqref{eqnSKomp}--\eqref{eqnSKompBC} with compactly supported initial data $n_0 \in L^1$,
    then
    \begin{equation}\label{eqnloss}
      N[n(T, \cdot)] + \int^T_0n(t,0)^2\ dt = N[n_0].
    \end{equation}
  \end{proposition}

  Physically, this means that photons can only be ``lost'' to the Bose-Einstein condensate, and not to infinity.
  Deferring the proof of Proposition~\ref{ppnLoss} to Section~\ref{sxnExistence}, we 
  now exhibit situations where the Bose-Einstein condensate must necessarily form in finite time.
  This is our next result, the proof of which is presented in Section~\ref{sxnMainProof}.

  \begin{corollary}\label{clyFTCondensate}
    Let $n$ be a non-negative entropy solution to~\eqref{eqnSKomp}--\eqref{eqnSKompBC} with initial data $n_0 \in L^1$.
    If $n_0$ is compactly supported and $N[n_0] > N[\hat n_0]$,
    there exists $T > 0$ such that
    \begin{equation*}
      N[n(T, \cdot)] < N[n_0].
    \end{equation*}
    In this situation the Bose-Einstein condensate necessarily forms in finite time.
  \end{corollary}

  In general, even though the system approaches one of the equilibria $\set{\hat n_\alpha}_{\alpha \in [0, 2]}$, we have no way of determining which one.
  We can, however, establish a non-zero lower bound on the total photon number in equilibrium.  Below, we use the notation $a\varmin b=\min(a,b)$
  and $a_+=\max(a,0)$.

  \begin{corollary}\label{clyNTFreeNCond}
    Let $n$ be a non-negative entropy solution to~\eqref{eqnSKomp}--\eqref{eqnSKompBC} with compactly supported initial data $n_0 \in L^1$.
    Let $\hat n_\alpha$ be the equilibrium solution for which~\eqref{eqnNtoNAlpha} holds.
    Then
    \begin{equation}\label{eqnNTFree}
      N[ \hat n_\alpha]
	\geq \sup_{t \geq 0} \int_0^2 \bigl(n(t, x) \varmin \hat n_0(x)\bigr) \, dx.
    \end{equation}
    Further, if $n_0$ is not identically $0$, neither is $\hat n_\alpha$.
  \end{corollary}

  The proof of Corollary~\ref{clyNTFreeNCond} requires a comparison principle which, for clarity of presentation, is also deferred to Section~\ref{sxnMainProof}.

  \subsection*{Plan of this paper}
  This paper is organized as follows.
  In Section~\ref{sxnSMainProof} we prove Theorem~\ref{thmMain} and Corollaries~\ref{clyFTCondensate} and \ref{clyNTFreeNCond}.
  Our proof relies on several lemmas and uses the notion of entropy solutions \'{a} la~\cite{Kruz}.
  Even though this is now standard, it involves a number of technicalities to adapt the results to the present situation.
  Thus for clarity of presentation, we define entropy solutions and prove Proposition~\ref{ppnLoss} (and the comparison and contraction lemmas) in Section~\ref{sxnExistence}.
  Finally in Section~\ref{sxnCompactness} we construct the appropriate ``sub'' and ``super''-solutions required to control the long time behavior of the system.

  \section{Proof of the main theorem}\label{sxnSMainProof}

  Our goal in this section is to prove the main theorem.
  The proof consists of several ingredients, some of which are technical.
  For clarity of presentation we briefly explain each part in a subsection below, and then prove Theorem~\ref{thmMain}.
  Due to its technical nature we postpone the definition and proof of existence of entropy solutions to Section~\ref{sxnExistence}.

  \subsection{Stationary solutions}
  We begin by computing the stationary solutions.
  \begin{lemma}\label{lmaStationary}
    All stationary entropy solutions to~\eqref{eqnSKomp} are given by~\eqref{eqnNAlpha} for some $\alpha \in [0, 2]$.
  \end{lemma}
  \begin{proof}
  Clearly if $n$ is a stationary solution to~\eqref{eqnSKomp}, then we must have $F(x, n) = c$ for some constant $c$.
  Since our boundary condition requires the incoming flux to vanish as $x \to \infty$, we must have $c = 0$.
  Thus looking for non-negative solutions to $F(x, n) = 0$ yields
  \begin{equation}\label{eqnStationary1}
    n(x) = 0
    \qquad\text{or}\qquad
    n(x) = (2x - x^2)_+\,.
  \end{equation}
  We show in Lemma~\ref{lmalipbound} that at points of discontinuity, entropy solutions (with compactly supported initial data) can \emph{only have upward jumps}.
  Combined with~\eqref{eqnStationary1} this immediately proves the lemma as desired.
  \end{proof}

  \subsection{Regularity of Entropy Solutions and Compactness}

  In the proof of Lemma~\ref{lmaStationary} we used the fact that entropy solutions can only have upward jumps.
  In fact, a much stronger result holds: the derivative of an entropy solution is bounded below, which leads to a BV estimate.
  Since this stronger fact will be used later, we state the lemmas leading to this result next.
  \begin{lemma}\label{lmaupperbound}
    Let $n$ be an entropy solution to~\eqref{eqnSKomp} with non-negative $L^1$ initial data which is supported on $[0, R]$ for some $R > 0$.
    Then
    \begin{equation}\label{eqnNooBound}
        \limsup_{t \to \infty} n(t, x) \leq \hat n_0(x) = (2x-x^2)_+\,,
    \end{equation}
    where $\hat n_0$ is the maximal equilibrium solution defined in~\eqref{eqnNAlpha} with $\alpha = 0$.%
    \end{lemma}
   \begin{lemma}\label{lmalipbound} 
    Let $n$ be an entropy solution to~\eqref{eqnSKomp} with non-negative $L^1$ initial data which is supported on $[0, R]$ for some $R > 0$.  Then for any $t>0$,
    a  one-sided Lipschitz bound holds: whenever  $0\leq x\leq y \leq R$,  
    \begin{equation}\label{eqnSupSol2}
      n(t, y) - n(t, x) \geq \underline{m}(t,R)(y - x)\,,
    \end{equation}
where $\underline{m}(t,R)<0$ depends only on $t$ and $R$
    and is increasing as a function of $t$.%
  \end{lemma}
  \begin{lemma}\label{lmabvbound}
    Let $n$ be a non-negative entropy solution to~\eqref{eqnSKomp}, with initial data supported on $[0,R]$ for some $R> 0$.
    Then~\eqref{eqnNormNBV} holds,
    and the trajectory $\set{n(t, \cdot)}_{t \geq 0}$ is relatively compact in $L^1$.%
  \end{lemma}

  The main idea behind the bounds~\eqref{eqnNooBound} and~\eqref{eqnSupSol2} is the construction of appropriate super-solutions.
  Once the bounds~\eqref{eqnNooBound} and~\eqref{eqnSupSol2} are established, the BV bound and $L^1$ compactness follow by relatively standard methods.
  For clarity of presentation we defer the proof of Lemmas~\ref{lmaupperbound}-\ref{lmabvbound} to Section~\ref{sxnCompactness}.

  \subsection{An \texorpdfstring{$L^1$}{L1} contraction estimate, and uniqueness}

  The next step in proving~\eqref{eqnNtoNAlpha} is to show that the $L^1$-distance between $n$ and every $n_\alpha$ cannot increase with time.
  We do this by proving a $L^1$ contraction estimate
  which also takes into account inflow and outflow flux.
  \begin{lemma}\label{lmaL1contract}
    Let $n$ and $n'$ be two non-negative bounded entropy solutions of~\eqref{eqnSKomp} with $L^1$ initial data.
    For any $R>2$, we have
    \begin{multline}\label{eqnL1Contract}
      \int^R_0 \abs{n(T,x)-n'(T,x)} \, dx
      + \int^T_0 \abs{ n(t,0)^2 - n'(t,0)^2 }\dt\\
      \leq\int^R_0\vert n(0,x)-n'(0,x)\vert\dx
      +\int^T_0 \abs{F(R,n(t,R))-F(R,n'(t,R))} \dt
    \end{multline}
  \end{lemma}
  Namely, the $L^1$ distance from $0$ to $R$ between two solutions, as well as between the outgoing fluxes at $x=0$, is controlled by the initial $L^1$ distance between the data and between the incoming fluxes at $x=R$.  We will later show that the incoming flux at $R$ vanishes, provided $R$ is greater than any value in the support of $n_0$.

  We also remark that $L^1$-contraction estimates 
  are well-known in the case when the flux $F$ is independent of position (see for instance~\cite{Bressan}).
  However, for a position dependent flux, this estimate is not easily found in the literature, and for completeness we present a proof in Section~\ref{sxnProofsOfLemmas}.

  As remarked above, in order to use Lemma~\ref{lmaL1contract}, we need to show that the incoming flux from the right vanishes at any point beyond the support of $n_0$.
  \begin{lemma}\label{lmaSupport}
    If $R \geq 2$ and $n$ is a non-negative entropy solution to~\eqref{eqnSKomp} with $L^1$ initial data that is supported in $[0, R]$,
    then for all $t \geq 0$ the function $n(t, \cdot)$ is also supported in $[0, R]$.
  \end{lemma}
  \begin{remark*}
    In fact, support of $n(t, \cdot)$ shrinks to the interval $[0, 2]$ as $t \to \infty$, as shown at the end of the proof of Lemma~\ref{lmaSupport}.  
  \end{remark*}
  The proof of Lemma~\ref{lmaSupport} is deferred to Section~\ref{sxnProofsOfLemmas}.
  We note, however, that Lemmas~\ref{lmaL1contract} and~\ref{lmaSupport} immediately yield uniqueness of entropy solutions to~\eqref{eqnSKomp}.
  \begin{proposition}[Uniqueness]\label{ppnUniqueness}
    If $n_0 \in L^1$ is compactly supported and non-ne\-ga\-tive, there is at most one non-negative entropy solution to~\eqref{eqnSKomp} with initial data~$n_0$.
  \end{proposition}
  \begin{proof}
    Let $n$ and $n'$ be two entropy solutions of~\eqref{eqnSKomp} with compactly supported $L^1$ initial data such that $n_0 = n_0'$.
    Using~\eqref{eqnL1Contract} we see
    \begin{equation*}
      \int_0^R \abs{n(T, x) - n'(T, x)} \, dx
      \leq -\int_0^T \abs{n(t, 0)^2 - n'(t, 0)^2}\, dt.
    \end{equation*}
    This is only possible if the right hand side vanishes, and hence $n = n'$ identically, proving uniqueness.
  \end{proof}

  \subsection{Proofs of the main results}\label{sxnMainProof}
  
  Using the above, we now prove the results stated in Section~\ref{sxnIntro}.
  We begin with the main theorem.

  \begin{proof}[Proof of Theorem~\ref{thmMain}]
    Uniqueness of non-negative entropy solutions was proved in Proposition~\ref{ppnUniqueness} above.
    For clarity of presentation, we address the existence of entropy solutions in Proposition~\ref{ppnExistence} in Section~\ref{sxnExistence} below.

    We next prove part~(\ref{thm1part1}) of the theorem. 
    First by Lemma~\ref{lmabvbound} we know $\set{n(t, \cdot)}_{t \geq 0}$ is relatively compact in $L^1$.
    Thus, to show that $n(t, \cdot)$ converges in $L^1$ as $t \to \infty$, it is enough to show that subsequential limits are unique.
    For this, let $(t_k)$ be a sequence of times $t_k \to \infty$ and $n_\infty \in L^1$ be such that $n(t_k, \cdot) \to n_\infty(\cdot)$ in $L^1$.
    We claim that $n_\infty$ is independent of the sequence $(t_k)$.
    Indeed, let
    \begin{equation*}
        C_\beta(t) = \int_0^\infty \abs{n(t, x) - \hat n_\beta(x)} \, dx,\qquad \beta\in[0,2].
    \end{equation*}
    For any $r < t$, Lemmas~\ref{lmaL1contract} and~\ref{lmaSupport} imply
    \begin{multline*}
      \int^R_0 \abs{n(t,x)-\hat n_\beta (x)} \, dx
      \leq
	\int^R_0 \abs{ n(r,x)-\hat n_\beta(x) } \, dx
	- \int_r^t \abs{n(s,0)^2- \hat n_\beta(0)^2 } \, ds
      \\
	= \int^R_0 \abs{ n(r,x)-\hat n_\beta(x) } \, dx - \int_r^t n(s, 0)^2 \, ds
    \end{multline*}
    and hence
    \begin{equation}\label{eqnCBetaDec}
      C_\beta(t) \leq C_\beta(r).
    \end{equation}
    Thus $C_\beta(t)$ must converge as $t \to \infty$, and we define
    \begin{equation*}
      \bar C_\beta = \lim_{t \to \infty} C_\beta(t).
    \end{equation*}
    By assumption, since $n(t_k, \cdot) \to n_\infty(\cdot)$ in $L^1$ we must also have
    \begin{equation*}
      \bar C_\beta = \int_0^\infty \abs{n_\infty - \hat n_\beta} \, dx.
    \end{equation*}
    Of course, $\bar C_\beta$ is independent of the sequence $(t_k)$, and so if we show that $n_\infty$ can be uniquely determined from the constants $\bar C_\beta$ we will obtain uniqueness of subsequential limits.

    To recover $n_\infty$ from $\bar C_\beta$, note that Lemma~\ref{lmaupperbound} implies $n_\infty(x) \leq \hat n_0(x)$ and hence
    \begin{equation*}
      n_\infty(x) \leq \hat n_\beta(x)
      \quad
      \text{for all } x > \beta \text{ and } \beta\in (0, 2].
    \end{equation*}
    Thus
    \begin{equation*}
      \bar C_\beta = \int_0^\beta n_\infty + \int_\beta^2 (\hat n_0 - n_\infty) \, dx,
    \end{equation*}
    and hence, for a.e.\ $\beta\in(0,2)$,
    \begin{equation*}
      2 n_\infty(\beta) = \hat n_0(\beta) + \partial_\beta \bar C_\beta ,
    \end{equation*}
    showing $n_\infty$ can be uniquely recovered from $\bar C_\beta$.
    This shows that $(n(t , \cdot)) \to n_\infty(\cdot)$ in $L^1$ as $t \to \infty$.
   We note that $n(t,x)$ is uniformly bounded for $t>t_0$ for each given $t_0>0$ as shown in the proof of Lemma~\ref{lmaupperbound} in Section~\ref{sxnupperboundproof}.  Using the $L^1$ convergence and standard dominated convergence arguments, we can pass to the limit $t\to\infty$ through the integrals in~\eqref{eqnweakform} and~\eqref{eqnweakentropy} to show that $n_\infty$ is a stationary entropy solution to~\eqref{eqnSKomp}.  Hence, there is some $\alpha\in(0,2)$ such that $n_\infty=\hat{n}_\alpha$.
    This proves~\eqref{eqnNtoNAlpha}.
    \smallskip

    Part~(\ref{thm1part2}) of the theorem asserts that the total photon number is a non-increasing function of time.
    To prove this observe that the total photon number is given by
    \begin{equation*}
      N[n(\cdot, t)] = \int_0^\infty n(t, x) \, dx
	= \int_0^\infty \abs[\big]{ n(t, x) - \hat n_2(x)} \, dx,
    \end{equation*}
    since $\hat n_2\equiv0$ and $n \geq 0$.
    By the $L^1$ contraction principle (Lemma~\ref{lmaL1contract}) the right hand side is a non-increasing function of time, and hence the same is true for the total photon number, finishing the proof.
  \end{proof}

  The loss formula~\eqref{eqnloss} uses techniques developed in the proof of Lemma~\ref{lmaL1contract} and we defer the proof to Section~\ref{sxnProofsOfLemmas}.
  Instead, we turn to Corollary~\ref{clyFTCondensate} and show that if we start with a total photon number larger than $N[\hat n_0]$, then the Bose-Einstein condensate must form in finite time.

  \begin{proof}[Proof of Corollary~\ref{clyFTCondensate}]%
    Using~\eqref{eqnNtoNAlpha} we see that
    \begin{equation*}
      \lim_{t \to \infty} N[n(t, \cdot)] = N[\hat n_\alpha] \leq N[\hat n_0].
    \end{equation*}
    Consequently if $N[n(0, \cdot)] > N[\hat n_0]$, then at some finite time $T$ we must have
    \begin{equation*}
      N[n(T, \cdot)]
	< \frac{ N[n(0, \cdot)] + N[\hat n_0] }{2}
	< N[n(0, \cdot)]
    \end{equation*}
    as desired.
  \end{proof}

  Finally, we prove Corollary~\ref{clyNTFreeNCond} and show that for any non-zero initial data, the equilibrium solution approached is not identically $0$.
  For this we need a comparison principle.

  \begin{lemma}[Comparison principle]\label{lmaCompare}
    Let $n$ and $n'$ be two non-negative entropy solutions to~\eqref{eqnSKomp} with compactly supported $L^1$ initial data $n_0$ and $n'_0$ respectively.  Then if $n_0(x)\leq n'_0(x)$ on $(0,\infty)$, then $n(t,x)\leq n'(t,x)$ on for any $t>0$, $x \in (0, \infty)$.
  \end{lemma}
  Relegating the proof of Lemma~\ref{lmaCompare} to Section~\ref{sxnExistence}, we prove Corollary~\ref{clyNTFreeNCond}.
  \begin{proof}[Proof of Corollary~\ref{clyNTFreeNCond}]
  For any $t_0>0$,  let $\underline{n}(t,x)$ be the solution of~\eqref{eqnSKomp}
  with initial data $\underline n(0,x)=n(t_0, x) \varmin \hat n_0(x)$. 
  Then the comparison principle (Lemma~\ref{lmaCompare}) immediately implies
  that for all $t,x\geq0$,
  \begin{equation*}
    \underline{n}(t,x) \leq n(t_0+t,x)\varmin\hat n_0(x).
  \end{equation*}
  Because $\underline{n}(t,0)=0$, as a consequence of Proposition~\ref{ppnLoss}
  the solution $\underline{n}$ conserves total photon number, therefore 
    \begin{equation*}
N[\hat n_\alpha] = \lim_{t\to\infty} N[n(t_0+t,\cdot)] \geq 
N[\underline n(t_0+t,\cdot)] = \int_0^2
\bigl(n(t_0, x) \varmin \hat n_0(x)\bigr)\,dx,
    \end{equation*}
 proving~\eqref{eqnNTFree}.

    It remains to show that the equilibrium solution~$\hat n_\alpha$ is not identically $0$, provided the initial data isn't either.
    For this, observe that if $N[n(t, \cdot)] = N[n_0]$ for all $t$, then $\int_0^2 \hat n_\alpha = N[n_0] > 0$, showing $n_\alpha$ is not identically $0$.
    Alternately, if $N[n(T, \cdot)] < N[n_0]$ at some finite time $T$, then by~\eqref{eqnloss} we must have $n(T', 0) > 0$ for some $T' \leq T$.
    Since the spatial discontinuities of $n$ can only be upward jumps (see Lemma~\ref{lmalipbound}), this forces
    \begin{equation*}
      \int_0^2 \bigl(n(T', x) \varmin \hat n_0(x) \bigr)\, dx > 0,
    \end{equation*}
    and~\eqref{eqnNTFree} now implies $\hat n_\alpha$ is not identically $0$.
  \end{proof}

  The remainder of this paper is devoted to proving Lemmas~\ref{lmaupperbound}--\ref{lmaSupport}, ~\ref{lmaCompare} and Proposition~\ref{ppnLoss}.

   \section{Entropy solutions}\label{sxnExistence}

  In this section, we define the notion of entropy solutions to~\eqref{eqnSKomp}-\eqref{eqnSKompBC} and prove existence as claimed in Theorem~\ref{thmMain}.
  We use the entropy introduced by Kruzkov~\cite{Kruz} which takes the family of convex functionals $\eta_k\defeq\vert n-k\vert$ as the entropies.
  
  \begin{definition}\label{dfnadmissible}
    We say that $n$ is an \textit{entropy  solution} to~\eqref{eqnSKomp}-\eqref{eqnSKompBC} if the following hold:
    \begin{asparaenum}
      \item
	The function $n \in L^1( [0, T] \times [0, \infty) ) \cap L^1( [0, T]; L^2( [0, \infty) ))$ and for each test function $\phi\in C^\infty_c((0,T)\times(0,\infty))$, we have the weak formulation
	\begin{equation}\label{eqnweakform}
	  \int_0^T\int_0^\infty
	    \paren[\big]{
	      n(t,x)\partial_t\phi(t,x)+F(x,n)\partial_x\phi(t,x)
	    }
	    \dx\dt=0.
	\end{equation}
      \item
	For any $k\in\mathbb{R}$ and non-negative test function $\phi\in C^\infty_c((0,T)\times(0,\infty))$, we have the \textit{Kruzkov entropy inequality}
	\begin{multline}\label{eqnweakentropy}
	  \int^T_0\int^\infty_0
	    \brak[\Big]{
	      \vert n(t,x)-k\vert\partial_t\phi+\sgn[n(t,x)-k][F(x,n(t,x))-F(x,k)]\partial_x\phi\\
	    -\sgn[n(t,x)-k]F_x(x,k)\phi}
	  \dx\dt\geq 0.
	\end{multline}
	\item The boundary condition~\eqref{eqnSKompBC} is satisfied in the $L^1$ sense, that is
	\begin{equation}\label{eqnentbc}
	\lim_{R\to\infty}\int^T_0\vert F(R,n(t,R))\vert~dt=0
	\end{equation}
	for any $T>0$.
    \end{asparaenum}
  \end{definition}
  \begin{remark}
    If $n$ is bounded and satisfies \eqref{eqnweakentropy}, then choosing
    \begin{equation*}
      k=\pm\sup_{[0,T]\times [0,\infty)}\vert n(t,x)\vert
    \end{equation*}
    shows that $n$ also satisfies \eqref{eqnweakform}.
  \end{remark}
 
\subsection{Contraction and Comparison}

  In this subsection, we prove a contraction and comparison principle for non-negative, compactly supported entropy solutions to~\eqref{eqnSKomp}.  Lemmas~\ref{lmaSupport} and \ref{lmaCompare} are proved by controlling $\vert n-n'\vert$ and $(n-n')_+$, or more generally $a\vert n-n'\vert+b(n-n')$ for some $a\geq 0$ and $b\in\mathbb{R}$.
  Our first lemma is the key step used to establish this.

\begin{lemma}\label{lmacontractcompare}
Let $a\geq 0, b\in\mathbb{R}$ and define $\Psi(s)\defeq a\vert s\vert+bs$.
Then for any two bounded entropy solutions to~\eqref{eqnSKomp} $n$ and $n'$,
\begin{multline}\label{eqnphidiff}
\int^T_0\int_0^\infty\Psi(n(t,x)-n'(t,x)) \partial_t \phi\\
+\Psi'(n(t,x)-n'(t,x))[F(x,n(t,x))-F(x,n'(t,x))] \partial_x \phi \, dx\, dt
    \geq 0
\end{multline}
    for any non-negative test function $\phi$.
\end{lemma}
\begin{remark}
  For Lemmas~\ref{lmacontractcompare} and~\ref{lmaPhiineq}, the entropy solutions do not need to be non-negative; non-negativity is not needed until Section~\ref{sxnProofsOfLemmas}.
\end{remark}

\begin{proof}

 To begin, we let $n$ and $n'$ be entropy solutions to~\eqref{eqnSKomp}.
   Take a smooth, non-negative function $g(t,x,s,y)$ from $(0,T)\times\mathbb{R}\times(0,T)\times\mathbb{R}$ and consider the weak entropy inequality~\eqref{eqnweakentropy} for $n(t,x)$.
   Fixing $s$ and $y$, we substitute $n'(s,y)$ for $k$ in the generalization of~\eqref{eqnweakentropy} and integrate over $s$ and $y$ to obtain

\begin{equation}\label{eqncombine1}
\begin{split}
\int^T_0\intzi\int^T_0\int^\infty_0&\vert n(t,x)-n'(s,y)\vert \partial_t g\\
&+\sgn[n(t,x)-n'(s,y)][F(x,n(t,x))-F(x,n'(s,y)] \partial_x g\\
&-\sgn[n(t,x)-n'(s,y)]\partial_x F(x,n'(s,y))g\dx\dt\dy\ds\geq 0.
\end{split}
\end{equation}

By repeating the procedure with the entropy solution $n'(s,y)$ with $n(t,x)$ serving the role of $k$, integrating over $t$ and $x$, and adding the result to~\eqref{eqncombine1} and multiplying by $a$, we obtain
\begin{equation}\label{eqncombine2}
\begin{split}
\int_0^T&\intzi\int_0^T\intzi a\vert n(t,x)-n'(s,y)\vert (\partial_t g+\partial_s g)\\
&+a\sgn(n(t,x)-n'(s,y))[F(x,n(t,x))-F(y,n'(s,y))](\partial_x g+\partial_y g)\\
&+a\sgn(n(t,x)-n'(s,y))\Bigl[
  \bigl( F(y,n'(s,y))-F(x,n'(s,y)) \bigr)\partial_x g\\
  &\qquad -\partial_x F(x,n'(s,y))g+(F(y,n(t,x))\\
  &\qquad -F(x,n(t,x)))\partial_y g+\partial_y F(y,n(t,x))g
\Bigr]\dx\dt\dy\ds\geq 0.
\end{split}
\end{equation}

Using $g$ as the test function in the weak formulations~\eqref{eqnweakform} for $n(t,x)$ and $n'(s,y)$, we integrate the weak formulation for $n$ over $s$ and $y$ and integrate the weak formulation for $n'$ over $t$ and $x$, multiply each by $b$, and add them together to obtain
\begin{multline}\label{weakcombine}
\int^T_0\int_0^\infty\int^T_0\int_0^\infty b\big(n(t,x)-n'(s,y)\big)(\partial_t g+\partial_s g)\\+b\big(F(x,n(t,x))-F(y,n'(s,y))\big)(\partial_x g+\partial_y g)~dx~dt~dy~ds=0.
\end{multline}
Adding~\eqref{weakcombine} to~\eqref{eqncombine2} and noting that $\Psi'(s)=a\sgn(s)+b$, we get
\begin{equation}\label{Phicombine}
\begin{split}
\int_0^T & \int_0^\infty\int_0^T\int_0^\infty \Psi(n(t,x)-n'(s,y))(\partial_t g+\partial_s g)\\
&+\Psi'(n(t,x)-n'(s,y))[F(x,n(t,x))-F(y,n'(s,y))](\partial_x g+\partial_y g)\\
&+a\sgn(n(t,x)-n'(s,y))\Big[
  \big( F(y,n'(s,y))-F(x,n'(s,y)) \big)\partial_x g\\
&\qquad-\partial_x F(x,n'(s,y))g+(F(y,n(t,x))\\
  &\qquad -F(x,n(t,x)))\partial_y g+\partial_y F(y,n(t,x))g
\Big]~dx~dt~dy~ds
\\&\defeq I_1+I_2+I_3\geq 0.
\end{split}
\end{equation}

We now take an arbitrary, non-negative test function $\phi(t,x)$ and define a sequence of non-negative test functions $\{g_h\}_{h>0}$ in terms of $\phi$ by
\begin{equation}\label{eqndefineg}
g_h(t,x,s,y) \defeq \phi\left(\frac{t+s}{2},\frac{x+y}{2}\right)\eta_h\left(\frac{t-s}{2}\right)\eta_h\left(\frac{x-y}{2}\right).
\end{equation}
Here $\eta_h$ is the approximate identity defined by
\[\eta_h(x) \defeq \frac{1}{h}\eta\left(\frac{x}{h}\right),\]
where $\eta\in C^\infty_c(\mathbb{R})$ is such that $\eta(x)=0$ for $\vert x\vert\geq 1$ and 
\[
  \int_\mathbb{R}\eta(x)\dx=1.
\]

We note that
\begin{subequations}
\begin{equation*}
(\partial_t g_h + \partial_s g_h)(t,x,s,y)= \partial_t \phi\left(\frac{t+s}{2},\frac{x+y}{2}\right)\eta_h\left(\frac{t-s}{2}\right)\eta_h\left(\frac{x-y}{2}\right)
\end{equation*}
\begin{equation*}
( \partial_x g_h+ \partial_y g_h)(t,x,s,y)= \partial_x \phi \left(\frac{t+s}{2},\frac{x+y}{2}\right)\eta_h\left(\frac{t-s}{2}\right)\eta_h\left(\frac{x-y}{2}\right).
\end{equation*}

\end{subequations}

Plugging this into~\eqref{Phicombine} and taking $h\to 0$, it is clear that $I_1+I_2$ converges to the left side of~\eqref{eqnphidiff}.  The proof that $I_3\to 0$ as $h\to 0$ follows from Taylor expansions and is done at the end of the proof of (3.12) in \cite{Kruz}.
\end{proof}

Next, we refine Lemma~\ref{lmacontractcompare} to control the difference between two solutions on a finite spatial domain.
\begin{lemma}\label{lmaPhiineq}
  Let $a\geq 0, b\in\mathbb{R}$ and define $\Psi(s)\defeq a\vert s\vert+bs$.
  Let $n$ and $n'$ be entropy solutions of~\eqref{eqnSKomp}.
  Then for any $C^1(0,\infty)$ curve $s(t)$,
\begin{equation}\label{curvel1}
  \begin{split}
    \int_0^{s(T)}&\Psi(n(T,x)-n'(T,x))\dx\\
 +\int^T_0&\Psi'(n(t,s(t))-n'(t,s(t)))[F(s(t),n(t,s(t)))-F(s(t),n'(t,s(t)))]\\
&-\dot{s}(t)\Psi(n(t,s(t))-n'(t,s(t)))\dt\\
   \leq\int^R_0&\Psi(n(0,x)-n'(0,x))\dx\\
    &+\int^T_0 \Psi'(n(t,0)-n'(t,0))[F(0,n(t,0))-F(0,n'(t,0))]\dt.
  \end{split}
  \end{equation}
In particular, if $s(t)\equiv R$, then
\begin{equation}\label{comparison2}
  \begin{split}
    \int^R_0\Psi&( n(T,x)-n'(T,x))\, dx
    \leq \int^R_0\Psi(n(0,x)-n'(0,x)) \, dx\\
      &- \int^T_0\Psi'(n(t,R)-n'(t,R))\left[F(R,n(t,R))-F(R,n'(t,R))\right] \, dt\\
      &+ \int^T_0\Psi'(n(t,0)-n'(t,0))\left[F(0,n(t,0))-F(0,n'(t,0))\right] \, dt
  \end{split}
\end{equation}
\end{lemma}
\begin{proof}
Here, we generalize the work in~\cite[Section 3]{Kruz}, in which a Lipschitz condition is assumed for the flux $F$.  However, for our model~\eqref{eqnSKomp}, we have no such condition.  Thus, we must retain some terms involving the flux, but will use properties of our particular flux to control these terms for non-negative entropy solutions when we go to prove the $L^1$ contraction and the comparison principle.

For this proof, we use the result of Lemma~\ref{lmacontractcompare} with an appropriate test function.  The test function we choose approximates the characteristic function of the time-space domain of integration $(0,T)\times(0,s(t))$ to mimic the use of the divergence theorem in $t$ and $x$ (see~\cite[Chapter 6]{Bressan}).
    To this end, we define $\rho$ and $\tau$ such that $0<\rho<\tau<T$.
    Let $\eps>0$ be small, and define the test function $\phi$ by
   \begin{equation}\label{chartest}
     \phi(t,x)=[\alpha_h(t-\rho)-\alpha_h(t-\tau)][\alpha_h(x-\eps)-\alpha_h(x-s(t)+\eps)]
   \end{equation}
  where
  \[
    \alpha_h(x) \defeq \int^x_{-\infty}\eta_h,
  \]
  and $s(t)$ is the curve denoting in the right side of the domain in the $xt$-plane.
  Thus,
  \begin{multline*}
    \partial_t \phi(t,x)=[\eta_h(t-\rho)-\eta_h(t-\tau)][\alpha_h(x-\eps)-\alpha_h(x-s(t)+\eps)]\\
    +\dot{s}(t)[\alpha_h(t-\rho)-\alpha_h(t-\tau)]\eta_h(x-s(t)+\eps)
  \end{multline*}
  and
  \begin{equation*}
    \partial_x \phi(t,x)=[\alpha_h(t-\rho)-\alpha_h(t-\tau)][\eta_h(x-\eps)-\eta_h(x-s(t)+\eps)].
  \end{equation*} 
  Note that as $h\to 0$, we have
\begin{subequations}
\begin{equation*}
  \partial_t \phi(t,x)\to\mathbbm{1}_{(\eps,s(t)-\eps)}(x)[\delta(t-\rho)-\delta(t-\tau)]+\dot{s}(t)\delta(x-s(t)+\eps)\mathbbm{1}_{(\rho,\tau)}(t)
\end{equation*}
\begin{equation*}
  \partial_x \phi(t,x)\to\mathbbm{1}_{(\rho,\tau)}(t)[\delta(x-\eps)-\delta(x-s(t)+\eps)]
\end{equation*}
\end{subequations}
where $\delta(\cdot)$ is the Dirac delta distribution and $\mathbbm{1}_{A}(\cdot)$ is the indicator function on the set $A$.
  
   Substituting the test function from~\eqref{chartest} into~\eqref{eqnphidiff} and taking $h\to 0$ yields
 \begin{equation}
 \begin{split}
\int_\eps^{s(\rho)-\eps}&\Psi(n(\rho,x)-n'(\rho,x))\dx-\int_\eps^{s(\tau)-\eps}\Psi(n(\tau,x)-n'(\tau,x))\dx\\
 +\int_\rho^\tau&\dot{s}(t)\Psi(n(t,s(t)-\eps)-n'(t,s(t)-\eps))\dt\\
 +\int^\tau_\rho&\Psi'(n(t,\eps)-n'(t,\eps))[F(\eps,n(t,\eps))-F(\eps,n'(t,\eps))]\dt\\
 -\int^\tau_\rho&\Psi'(n(t,s(t)-\eps)-n'(t,s(t)-\eps)0\cdot\\
&[F(s(t)-\eps,n(t,s(t)-\eps))-F(s(t)-\eps,n'(t,s(t)-\eps))]\dt\geq 0.
\end{split}
 \end{equation}
    Taking $\rho\to 0$, $\tau\to T$, and $\eps\to 0$ gives~\eqref{curvel1}.  Taking $s(t)\equiv R$ in~\eqref{curvel1} gives~\eqref{comparison2}, completing the proof.
\end{proof}

\subsection{Proofs of Proposition~\ref{ppnLoss} and Lemmas~\ref{lmaL1contract}, \ref{lmaSupport} and~\ref{lmaCompare}}\label{sxnProofsOfLemmas}



The $L^1$-contraction (Lemma~\ref{lmaL1contract}) follows immediately from Lemma~\ref{lmaPhiineq} and we address this first.

\begin{proof}[Proof of Lemma~\ref{lmaL1contract}]
  Choosing $a=1$ and $b=0$ in Lemma~\ref{lmaPhiineq}, we see $\Psi(s) = \abs{s}$, and Lemma~\ref{lmaL1contract} immediately follows from~\eqref{comparison2} and the fact that $n$ and $n'$ are non-negative.
\end{proof}

We now turn to showing Lemma~\ref{lmaSupport}, that if a non-negative entropy solution has compact support initially, then it will have compact support for all time.
\begin{proof}[Proof of Lemma~\ref{lmaSupport}]
We use an analog of~\eqref{curvel1} where we take $s(t)$ to be the left boundary of the spatial domain, and take $n'(t,x)\equiv 0$, which is clearly a non-negative entropy solution to~\eqref{eqnSKomp}.  Thus, we get
\begin{multline}\label{compactsupineq}
\int_{s(T)}^\infty\vert n(T,x)\vert\dx
-\int^T_0\vert n(t,s(t))\vert\left(2s(t)-s^2(t)-n(t,s(t))-\dot{s}(t)\right)\dt\\
\leq\int_{s(0)}^\infty\vert n(0,x)\vert\dx
\end{multline}
where we have used~\eqref{eqnentbc} to eliminate the integral with the right-side flux terms.  Setting $s(t)\equiv R$, where $R\geq 2$ is an upper bound for the support of $n_0$, in~\eqref{compactsupineq} and using the fact that $n$ is a non-negative entropy solution complete the proof.

In fact, we can strengthen the result in Lemma~\ref{lmaSupport}.  It is clear that the equations for the characteristics of~\eqref{eqnSKomp} are
\begin{gather*}
  \dot{x}=2x-x^2-2n\\
  \dot{n}=2xn-2n.
\end{gather*}
We can see that on a characteristic, if $n$ starts with a value of $0$, it will remain $0$ on the characteristic.  This observation allows us to strengthen our result from Lemma~\ref{lmaSupport}.   For characteristics starting at values $x_0>R$ outside the support of $n_0$, we will have
\begin{equation}\label{eqnlargechar}
\dot{x}=2x-x^2
\end{equation}
with $n(t,x(t))=0$ along the characteristics given by~\eqref{eqnlargechar}.
We also note that the equation for the characteristic is a logistic equation, so if $R>2$, then as $t\to\infty$, $x$ decreases to $2$ along the characteristic starting at $x_0$.  Define $s$ such that $\dot{s}(t)=2s(t)-s^2(t)$ and $s(0)=R$, where $[0,R]$ is the support of $n_0$. Substituting this into~\eqref{compactsupineq}  we obtain 
\begin{equation}\label{eqncmpsupp}
\int_{s(T)}^\infty\vert n(T,x)\vert\dx+\int^T_0 n^2(t,s(t))\dt\leq\int_{R}^\infty\vert n(0,x)\vert\dx = 0,
\end{equation}
finishing the proof.
\end{proof} 

\begin{remark}
  Since $s(T) \to 2$ as $T\to\infty$, the above proof shows that the support of the $n$ shrinks to $[0, 2]$ as $T \to \infty$, as remarked earlier.
\end{remark}

We now prove the comparison principle.

\begin{proof}[Proof of Lemma~\ref{lmaCompare}]
  Choosing $a=b=\frac{1}{2}$, the function $\Psi$ defined in Lemma~\ref{lmacontractcompare} becomes the positive part (i.e.\ $\Psi(s)=s_+$).
  Using~\eqref{comparison2}, we take $R$ to be an upper bound of the supports of $n_0$ and $n'_0$, obtaining
  \begin{equation}
    \begin{split}
      &\int^R_0\Psi(n(T,x)-n'(T,x))~dx\\
      &\leq \int^R_0\Psi(n(0,x)-n'(0,x))~dx\\
      &- \int^T_0\Psi'(n(t,R)-n'(t,R))\left[F(R,n(t,R))-F(R,n'(t,R))\right]~dt\\
      &+ \int^T_0\Psi'(n(t,0)-n'(t,0))\left[F(0,n(t,0))-F(0,n'(t,0))\right]~dt,
    \end{split}
  \end{equation}
   Using Lemma~\ref{lmaSupport}, and noting that $n$ and $n'$ are non-negative, and using the definition of $F$, we obtain that 
   \begin{equation}
   \int^R_0(n(T,x)-n'(T,x))_+~dx\leq\int^R_0(n(0,x)-n'(0,x))_+~dx
   \end{equation}
  which immediately yields the result.
\end{proof}

Finally, we conclude this subsection by proving Proposition~\ref{ppnLoss}, using techniques used in the proof of Lemma \ref{lmaL1contract}.
\begin{proof}[Proof of Proposition~\ref{ppnLoss}]
  In the weak formulation~\eqref{eqnweakform} we use the test function from~\eqref{chartest} with $s(t)\equiv R$ where $R$ is an upper bound for the support of $n_0$.
  Following the same argument as that used in proving~\eqref{curvel1}, we substitute $\phi$ into~\eqref{eqnweakform} and obtain~\eqref{eqnloss}, using Lemma~\ref{lmaSupport} to show that the term $\int^T_0F(R,n(t,R))\ dt=0$.
\end{proof}

\subsection{Existence}  
We devote this section to proving the existence of entropy solutions.
Kruzkov~\cite[Sections 4 and 5]{Kruz} uses a vanishing viscosity argument to show that entropy solutions for the Cauchy problem on all of $\mathbb{R}^n$ exist, provided the flux $F = F(p, q)$,
  $\partial_qF$, $\partial_q\partial_pF$, and $\partial^2_pF$ are all continuous, $\partial_pF(p,0)$ is bounded, $\partial_qF$ is bounded on horizontal strips (i.e. domains where $q$ is bounded), and $-\partial_p\partial_qF$ is bounded above on horizontal strips.
  If we naively extend our problem on $\R^+$ to the Cauchy problem on $\mathbb{R}$, it is clear that we meet all of these requirements except the boundedness of $\partial_qF$ and $-\partial_p\partial_qF$.

\begin{proposition}\label{ppnExistence}
  Let $n_0\in L^1(\mathbb{R}^+)$ be non-negative with compact support on some subset of $[0,R]$ for some $R>2$.
  Then there exists a non-negative entropy solution to~\eqref{eqnSKomp}-\eqref{eqnSKompBC} in the sense of Definition~\ref{dfnadmissible}.
\end{proposition}
\begin{proof}
We prove the existence of entropy solutions by using a vanishing viscosity argument.  We consider the problem
\begin{equation}\label{artvis}
\partial_t n_\eps+\partial_x\bar{F}(x,n_\eps)=\eps\partial_x^2n_\eps
\end{equation}
on the entire real line and will consider the vanishing viscosity limit $\eps\downarrow 0$.  We consider the Cauchy problem with $L^1(\mathbb{R})$ initial data \[n^0(x)=\begin{cases} n_0(x), & x>0\\ 0, & x\leq 0\end{cases}.\]  The key step is extending the flux $F$ on $(0,\infty)\times\mathbb{R}$ to some flux $\bar{F}$ on $\mathbb{R}\times\mathbb{R}$ so that $\bar{F}$ meets the boundedness and regularity requirements listed above.  In light of Lemma~\ref{lmaSupport}, we know that for any time $t>0$, any non-negative entropy solution will be zero at $x>R$.  Thus, we fix the value of the flux at $x=R$, and extend this rightward toward infinity.  We will also extend the flux leftward to and obtain
\begin{equation}\label{eqnExtendedFlux}
      \bar{F}(x, n) = g(x) n - n^2
    \end{equation}
    where $g$ is a smooth function such that
    \begin{equation*}
      g(x) =
	\begin{dcases}
	  2x - x^2  & x \in [0, R]\\
	  2R - R^2  - 1 & x > 2R\\
	  -1  & x < -1.
	\end{dcases}
   \end{equation*}
extending smoothly in $x$ using standard techniques.  A simple calculation shows that $\bar{F}_n$ is Lipschitz in $n$ with a Lipschitz constant of $2$.  Using standard parabolic existence results and a standard parabolic comparison principle
argument, as a simple calculation shows the linear part of the parabolic equation~\eqref{artvis} is bounded below, it is clear that for each $\eps>0$, there is some non-negative $\tilde{n}_\eps$ solving~\eqref{artvis}.  Using standard techniques for taking the vanishing viscosity limit (see, for example, \cite[Section 4]{Kruz} and~\cite[Chapter 6]{Dafermos}) we take $\eps\to 0$, we obtain the existence of a non-negative entropy solution $\tilde{n}$ to the Cauchy problem
\begin{equation}
\begin{cases}
\partial_t\tilde{n}+ \partial_x \bar{F}(x,\tilde{n})=0 & (t,x)\in(0,\infty)\times\mathbb{R}\\
n(0,x)=n^0(x)& x\in\mathbb{R}.
\end{cases}
\end{equation}  
It is left to restrict the problem to the half-line.
Adapting the proof of Lemma~\ref{lmaSupport} is one can quickly show $\tilde{n}(t,x) = 0$ for any $x>R$, $t>0$.
  Thus, we can restrict the class of test functions we consider for the weak formulation to be those compactly supported on $(0,\infty)$, and therefore obtain the entropy inequality on the half-line.  Setting $n$ to be the restriction of $\tilde{n}$ on the half-line completes the proof.
\end{proof}

  \section{Regularity of Entropy Solutions and Compactness}\label{sxnCompactness}
  In this section we prove
  Lemmas~\ref{lmaupperbound}--\ref{lmabvbound},
  concerning the BV estimates for entropy solutions to~\eqref{eqnSKomp}.
  In each of the following subsections, we prove each lemma in turn.
  \subsection{A sharp upper bound as \texorpdfstring{$t \to \infty$}{t to infinity}}\label{sxnupperboundproof}

  In this section, we prove the boundedness of entropy solutions by the maximal supersolution from~\eqref{eqnNAlpha}.
  \begin{proof}[Proof of Lemma~\ref{lmaupperbound}]
    The main idea behind the proof is to find a special function $\bar n$ such that
    \begin{equation}\label{eqnSKompSupSol}
      \partial_t \bar n + \partial_x \bar F \geq 0.
    \end{equation}
    This heuristically corresponds to the notion of a super-solution.
    In the context of parabolic equations the comparison principle guarantees that a solution that starts below a super-solution will always stay below.
    For hyperbolic conservation laws, however, there isn't an analogous result as the notion of entropy super-solutions has not been developed rigorously.
    A comparison principle is known (see~\cite[Theorem 3]{Kruz}), but only compares two entropy solutions.
    We circumvent the use of a comparison principle for entropy super-solutions by using viscous limits.

    Before delving into the technical details, we begin with a formal computation of a special ``super-solution''.
    Choose $\bar n$ to be a function that satisfies
    \begin{equation}\label{eqnBarNdef}
        F( x, \bar n ) =(2x-x^2)\bar n-\bar n^2 = - K(t) G(x),
    \end{equation}
    where $K$ and $G$ are chosen in order to arrange~\eqref{eqnSKompSupSol}.

    Solving~\eqref{eqnBarNdef} explicitly (with the constraint $\bar n \geq 0$) gives
    \begin{equation}\label{eqnBarNExplicit}
      \bar n( x, t )
	= \frac{1}{2} \paren[\Big]{
	    g +  \sqrt{ g^2 + 4KG }
	  },
    \end{equation}
    where
    \begin{equation*}
      g(x) = 2x - x^2.
    \end{equation*}
    We compute
    \begin{equation*}
      \partial_t \bar n + \partial_x \bar F
	\geq \frac{G \partial_t K}{\sqrt{g^2 + 4KG}} - K \partial_x G
	\geq \frac{\sqrt{G} \partial_t K}{2\sqrt{K}} - K \partial_x G,
    \end{equation*}
    provided $K$ is chosen such that $\partial_t K \leq 0$.
    Since $K$ only depends on $t$ and $G$ only depends on $x$, we separate variables to ensure the right hand side of the above vanishes.
    This gives
    \begin{equation*}
      K(t) = \frac{1}{\paren{\beta t + c_1}^2}
      \quad\text{and}\quad
      G(x) = \frac{\beta^2 \paren{R + c_2 - x}^2}{4},
    \end{equation*}
    where $\beta$, $c_1$ and $c_2$ are non-negative constants.
    Choosing $c_1 = 0$ small and $c_2 > 0$ provides a ``super-solution'' with initial data $\bar n_0 = \infty$.
    If a notion of entropy super-solutions and corresponding comparison principle was available, we would have
    \begin{equation*}
      n(t, x) \leq \bar n(t, x) \xrightarrow{t \to \infty} g(x)_+ = \hat n_0(x),
    \end{equation*}
    proving~\eqref{eqnNooBound} as desired.

    Proceeding to the actual proof, we avoid the above difficulty by using viscous limits.
    Recall $n$ is the pointwise limit of $n_\epsilon$, where $n_\epsilon$ solves
    \begin{equation}\label{eqnNep}
      \partial_t n_\epsilon + \partial_x \bar F (x, n_\epsilon) = \epsilon \partial_x^2 n_\epsilon,
    \end{equation}
    on the whole line $x \in \R$ and vanishes at infinity.
    Here $\bar F$ is the extended flux
    defined by
    \begin{equation*}
      \bar F(x, n) = g(x) n - n^2
    \end{equation*}
    where $g$ is a smooth function such that
    \begin{equation}\label{gdef}
      g(x) =
	\begin{dcases}
	  2x - x^2  & x \in [0, R]\\
	  2R - R^2  - 1 & x > 2R\\
	  -1  & x < -1.
	\end{dcases}
    \end{equation}

    We first claim that for any fixed $\delta > 0$, the functions $n_\epsilon$ converge uniformly 
(as $\epsilon \to 0$)
    to $0$ on the set 
    \begin{equation*}
      \set{ t \geq 0, \; x \not\in [-1-\delta, R+\delta] }.
    \end{equation*}
    To see this, we note that the parabolic comparison principle can be used on equation~\eqref{eqnNep} (see the proof of Proposition~\ref{ppnExistence}).
    For simplicity, assume that the function $g$ is chosen so that $g' \leq 0$ on $[R, \infty)$.
    In this case, a function $m_\epsilon$ that only depends on $x$ and satisfies
    \begin{equation*}
      \partial_x m_\epsilon \leq 0
      \quad\text{and}\quad
      -\norm{g}_\infty \partial_x m_\epsilon - 2 m_\epsilon \partial_x m_\epsilon - \epsilon \partial_x^2 m_\epsilon = 0
    \end{equation*}
    is clearly a super-solution to~\eqref{eqnNep} on the interval $[R, \infty)$.
    Solving this equation with boundary conditions $m(R) = \infty$ and decay at infinity yields
    \begin{equation*}
      m_\epsilon(x) = 
	\frac{
	  \norm{g}_\infty
	  \exp\paren[\Big]{ -\norm{g}_\infty \paren[\Big]{\dfrac{x - R}{\epsilon}} }
	}{
	  1 - \exp\paren[\Big]{ -\norm{g}_\infty \paren[\Big]{\dfrac{x - R}{\epsilon}} }
	}.
    \end{equation*}

    Since $n_\epsilon(0, x) = 0$ for $x \geq R$ and $n_\epsilon(t, R) < m_\epsilon(R)$ the comparison principle guarantees
    \begin{equation}\label{eqnNepLeqMep}
      n_\epsilon(x, t) \leq m_\epsilon(x)
      \quad\text{for $x \geq R$ and $t \geq 0$.}
    \end{equation}
    This shows that as $\epsilon\to0$, $n_\epsilon \to 0$ uniformly on $\set{ t \geq 0, \; x \geq R+\delta }$.
    A similar argument can be applied to obtain uniform convergence on $\set{ t \geq 0, \; x \leq -\delta }$, proving the claim.

    Now, suppose momentarily that $\norm{n_0}_\infty < \infty$.
    For $M > 0$, define the functions $K_M$ and $G$ by
    \begin{equation*}
      K_M(t) = \frac{1}{\paren[\big]{3t + \frac{1}{M}}^2}
      \quad\text{and}\quad
      G(x) = \paren{3R - x}^2,
    \end{equation*}
    and define $\bar n_M$ by~\eqref{eqnBarNExplicit}.
    We compute
    \begin{equation*}
      \partial_t \bar n_M + \partial_x \bar F - \epsilon \partial_x^2 \bar n_M
	\geq \frac{\sqrt{G} \partial_t K_M}{2\sqrt{K_M}} - K_M \partial_x G
	  - \epsilon \partial_x^2 \bar n_M
	= \sqrt{G} K_M - \epsilon \partial_x^2 \bar n_M.
    \end{equation*}
    For any fixed $T > 0$ observe
    \begin{equation*}
      \inf_{\substack{t \in [0, T]\\x \in [-1, 2R]} } \sqrt{G} K_M > 0
      \quad\text{and}\quad
      \sup_{\substack{t \in [0, T]\\x \in [-1, 2R]} } \abs{\partial_x^2 \bar n_M} < \infty.
    \end{equation*}
    Thus for $\epsilon$ small enough we have
    \begin{equation*}
      \partial_t \bar n_M + \partial_x \bar F - \epsilon \partial_x^2 \bar n_M \geq 0
    \end{equation*}
    on $t \in [0, T]$ and $x \in [-1, 2R]$.
    Since we have (temporarily) assumed $\norm{n_0}_\infty < \infty$, we can make $M$ large enough to ensure
    \begin{equation*}
      \Chi*{[0, R]}(x) n_0(x) = n_\epsilon(0, x) \leq \inf_{x \in [-2, 2R]} \bar n_M(x).
    \end{equation*}
    Finally, for the boundary conditions let $S =\set{t \in [0, T], x \in \set{-1, 2R} }$.
    Since $n_\epsilon \to 0$ uniformly on $S$, and $\inf_S \bar n_M > 0$ for $\epsilon$ small enough we must have $\bar n_M \geq n_\epsilon$ on~$S$.

    Thus, the parabolic comparison principle guarantees
    \begin{equation*}
      n_\epsilon \leq \bar n_M
      \quad\text{for } t\in [0, T],~x \in [-1, 2R].
    \end{equation*}
    Sending $\epsilon \to 0$ and $M, T \to \infty$ now yields
    \begin{equation}\label{eqnNbound}
      n(x, t) \leq \bar n(x, t)
      \quad\text{for } t \geq 0, x \in [0, R],
    \end{equation}
    where $\bar n \defeq \lim \bar n_M$ as $M \to \infty$.

    The above was proved with the assumption that $\norm{n_0}_\infty< \infty$.
    However, since the right hand side is independent of $\norm{n_0}_\infty$, we can immediately dispense with this assumption.
    Indeed, choose a sequence of non-negative $L^\infty$ that converge to $n_0$ in $L^1$.
    Then the corresponding solutions each satisfy the bound~\eqref{eqnNbound}, and converge to $n$ in $L^1$.
    Hence $n$ itself must satisfy~\eqref{eqnNbound}.
    Finally, sending $t \to \infty$ in~\eqref{eqnNbound} proves~\eqref{eqnNooBound}, concluding the proof.
  \end{proof}

  \subsection{A one-sided Lipschitz bound}\label{sxnLipBound}
  We now turn to proving the one-sided Lipschitz bound on entropy solutions.
  \begin{proof}[Proof of Lemma~\ref{lmalipbound}]
    As before, we use the fact that $n$ is the pointwise limit of the viscous solutions~$n_\epsilon$, where~$n_\epsilon$ solves~\eqref{eqnNep} on the whole line $x \in \R$ and vanishes at infinity.
    Let $m_\epsilon = \partial_x n_\epsilon$, and we compute
    \begin{equation}\label{eqnMep}
      \partial_t m_\epsilon
	- 2m_\epsilon^2
	+ 2g' m_\epsilon
	+ g'' n_\epsilon
	+ (g - 2n_\epsilon) \partial_x m_\epsilon
	- \epsilon \partial_x^2 m_\epsilon
	= 0.
    \end{equation}
    The first step in the proof is to bound $m_\epsilon$ from below.
    We do this by constructing a sub-solution $\underline{m}_\epsilon$ that only depends on time.

    To find $\underline{m}_\epsilon$, observe that if $\underline{m}_\epsilon \leq 0$ then
    \begin{multline}\label{eqnMLowerBar1}
      \partial_t \underline{m}_\epsilon
	- 2\underline{m}_\epsilon^2
	+ 2g' \underline{m}_\epsilon
	+ g'' n_\epsilon
	+ (g - 2n_\epsilon) \partial_x \underline{m}_\epsilon
	- \epsilon \partial_x^2 \underline{m}_\epsilon
      \\
      \leq
	\partial_t \underline{m}_\epsilon
	  - 2\underline{m}_\epsilon^2
	  - C' \underline{m}_\epsilon
	  + \sup (g'' n_\epsilon).
    \end{multline}
    where
    \begin{equation*}
      C' \defeq 2\norm{g'}_\infty
      \quad\text{and}\quad
      C_\epsilon \defeq \sup ( g'' n_\epsilon ) .
    \end{equation*}
    Note $n_\epsilon \geq 0$ on $\R$ and for some small $\delta > 0$ we must have $g'' < 0$ on $[-\delta, R + \delta]$.
    Further, since $n_\epsilon \to 0$ uniformly on $\set{ t \geq 0, x \not\in [-\delta, R + \delta ] }$, we must have $C_\epsilon \to 0$ as $\epsilon \to 0$.

    Now equating the right hand side of~\eqref{eqnMLowerBar1} to $0$ and solving for $\underline{m}_\epsilon$ yields
    \begin{equation*}
      \underline{m}_\epsilon(t)
	= \frac{-C'}{4}
	  - \frac{\sqrt{8 C_\epsilon + (C')^2}
	      \paren[\Big]{
		1 + \exp\paren[\big]{ - t \sqrt{ 8C_\epsilon + (C')^2 } }
	      }
	    }{
	      4
	      \paren[\Big]{
		1 - \exp\paren[\big]{ - t \sqrt{ 8C_\epsilon + (C')^2 } }
	      }
	    }.
    \end{equation*}
    Since $\underline{m}_\epsilon(0) = -\infty$, and (by construction) $\underline{m}_\epsilon$ is a sub-solution to~\eqref{eqnMep}, 
provided that $\epsilon$ is small enough to guarantee $8 C_\epsilon \leq 1$,
    we must have
    \begin{equation*}
      \partial_x n_\epsilon
	= m_\epsilon
	\geq \underline{m}_\epsilon
	\geq \underline{m}\,,
    \end{equation*}
    where 
    \begin{equation}\label{mdef}
\underline{m}=      \underline{m}(t,R)
	\defeq
	\frac{-C'}{4}
	  - \frac{\sqrt{1 + (C')^2} }{
	      2
	      \paren[\Big]{
		1 - \exp\paren[\big]{ - t \sqrt{ 1 + (C')^2 } }
	      }
	    }.
    \end{equation}
    Now for $0 \leq x < y \leq R$ and $t > 0$ we have
    \begin{equation*}
      n(t, y) - n(t, x)
	= \lim_{\epsilon \to 0} n_\epsilon(t, y) - n_\epsilon(t, x)
	\geq \underline{m}(t,R) (y - x)
    \end{equation*}
    concluding the proof.
  \end{proof}

  \subsection{Compactness in \texorpdfstring{$L^1$}{L1}}\label{sxnL1Compactness}
We conclude this paper with the proof of Lemma~\ref{lmabvbound}.
\begin{proof}[Proof of Lemma~\ref{lmabvbound}]
We fix $t_0>0$ and let $n$ be a non-negative entropy solution with initial data supported on $[0,R]$, with $R\geq 2$.
Non-negativity and Lemma~\ref{lmaPhiineq} with $n' = 0$ (or Proposition~\ref{ppnLoss}) imply $\|n(t,\cdot)\|_{L^1(0,\infty)}$ is a non-increasing function of time, and thus is bounded by  $\norm{n_0}_{L^1}$.
Thus, we need to only control the total variation of $n(t,\cdot)$ to control the BV norm.

  For this, observe that Lemma~\ref{lmaSupport} guarantees $n$ is supported on $[0,R]$ for all $t>0$, so it suffices to restrict our attention to $[0,R]$.
By the one-sided Lipschitz bound of Lemma~\ref{lmalipbound}, we can write
$n(t,\cdot)$ as the difference of two increasing functions as
\[
n(t,x)= (n(t,x)-\underline{m}(t,R)x ) + \underline{m}(t,R)x\;.
\]
Therefore, because $n(t,0)\geq 0$ and $n(t,R)=0$, we deduce that for all $t\geq t_0$,
   \begin{equation*}
    \operatorname{TV}[n(\cdot , t)]
      \leq 2  \abs{\underline{m}(t_0,R)} R \;,
  \end{equation*}
%
  which immediately implies~\eqref{eqnNormNBV}.
  Finally, the result of Helly (see~\cite[Theorem 2.3]{Bressan}) shows relative compactness of $\set{n(t, \cdot)}_{t > 0}$ in $L^1$ completing the proof.
 \end{proof}

  \bibliographystyle{siamplain}
  \bibliography{mainbib,refs}
\end{document}